\documentclass[11pt]{amsart}

\title{Recognizing bicoset digraphs which are $X$-joins and automorphism groups of bicoset digraphs}

\author{Rachel Barber\\
Department of Mathematics, Hood College\\ 
401 Rosemont Ave.\\
Frederick, MD 21701\\
\\
and\\
\\
Ted Dobson \\
IAM and FAMNIT, University of Primorska\\
Muzejska trg 2\\
Koper 6000, Slovenia\\
\\
and\\
\\
Gregory Robson\\
FAMNIT, University of Primorska\\
Muzejska trg 2\\
Koper 6000, Slovenia\\
}

\usepackage{amssymb,latexsym,amsmath,amsthm}
\usepackage{graphicx}
\usepackage{hyperref}
\usepackage{cleveref}
\usepackage[justification=centering]{caption}
\usepackage{mathrsfs}
\usepackage{ulem}
\usepackage[all,arc,curve,color,frame,pdf,line]{xy}
\usepackage{comment}
\usepackage{enumitem}

\usepackage{tikz}
\usetikzlibrary{graphs,graphs.standard,quotes}
\usetikzlibrary{arrows.meta}
\tikzstyle{V}=[draw, fill =black, circle, inner sep=0pt, minimum size=4pt]

\usepackage[margin=1.0in]{geometry}
\usepackage{fancyhdr}
\usepackage{lipsum}

\theoremstyle{plain}
\numberwithin{equation}{section}
\newtheorem{theorem}{Theorem}[section]
\newtheorem{corollary}[theorem]{Corollary}

\newtheorem{lemma}[theorem]{Lemma}

\theoremstyle{definition}
\newtheorem*{solution*}{Solution}
\newtheorem{definition}[theorem]{Definition}

\newtheorem{example}[theorem]{Example}
\newtheorem{problem}[theorem]{Problem}

\newtheorem{note}[theorem]{Notation}

\DeclareMathOperator{\mP}{\mathcal{P}}

\def\Aut{{\rm Aut}}
\def\Haar{{\rm Haar}}
\def\tl{\triangleleft}
\def\wh{\widehat}
\def\la{{\langle}}
\def\ra{{\rangle}}
\def\Z{{\mathbb Z}}
\def\cal{\mathcal}

\def\core{{\rm core}}
\def\Cos{{\rm Cos}}

\begin{document}
\thanks{The work of the second author is supported in part by the  Slovenian Research Agency (research program P1-0285 and research projects N1-0140, N1-0160, J1-2451, N1-0208, J1-3001, J1-3003, J1-4008, and J1-50000), while the work of the third author is supported in part by the Slovenian Research Agency (research program P1-0285 and Young Researchers Grant).}

\begin{abstract}
We examine bicoset digraphs and their natural properties from the point of view of symmetry.  We then consider connected bicoset digraphs that are $X$-joins with collections of empty graphs, and show that their automorphism groups can be obtained from their natural irreducible quotients.  We then show that such digraphs can be recognized from their connection sets.
\end{abstract}

\maketitle

The oldest and most popular construction of vertex-transitive graphs is that of a Cayley graph.  The basic idea is that, given a group $G$ and a subset of $G$, one constructs a graph with a large amount of symmetry.  There are now a variety of graphs and digraphs that have similar constructions, in that they also use a group $G$ and a subset $S$ of $G$, and produce a graph or digraph with a large amount of symmetry.  The differences between the constructions are usually in how the group acts.  Today, there are Cayley digraphs, double-coset digraphs, Haar graphs, bicoset graphs, and digraphs that are all constructed in this manner.  The work in this paper is the third in an unplanned series of papers that investigate what is in hindsight an obvious question.  Given a group $G$ and subset $S$, what symmetry information can one obtain from a Cayley digraph about the symmetries of the other graphs and digraphs constructed with the different actions of $G$?

The first paper in this sequence, motivated by an isomorphism problem (the so-called BCI-problem), is \cite{Dobson2022}, where a relationship was given between the automorphism groups of the Cayley digraphs of abelian groups and their corresponding Harr graphs.  The second paper \cite{BarberDpreprint}, motivated by the lack of a definition of a ``generalized wreath product" for Cayley graphs of nonabelian groups, gave a recognition theorem for Cayley and coset digraphs that can be written as a wreath product from the set $S$ used to construct them.  Fortuitously, the recognition theorem is exactly what is needed in order to give the relationship between the automorphism groups of Cayley digraphs and coset digraphs.  It turns out that, while different, the automorphism group of one can be obtained from the automorphism group of the other \cite[Theorem 3.3]{BarberDpreprint}.  We should also mention that every vertex-transitive digraph is isomorphic to a coset digraph, and so this reduces the problem of finding the automorphism group of vertex-transitive digraphs to the problem of finding the automorphism groups of Cayley digraphs.

In this, the third paper in the series, we prove in \Cref{bicosjoin} an analogue of the recognition theorem mentioned in the previous paragraph for Haar graphs, digraphs and bicoset graphs and digraphs, the bipartite analogues of Cayley graphs, digraphs, and double coset graphs and digraphs.  For these graphs, the action of $G$ that is used need not be transitive, and so a generalization of the wreath product, the $X$-join, will play the role that the wreath product played in \cite{BarberDpreprint}.  Although Hemminger \cite{Hemminger1968} determined a considerable amount of information on the automorphism group of the $X$-join, he only considered graphs (and we need that information for digraphs as well), and it is actually easier to derive the information about the $X$-join we will need in this paper independently, which we do in Section 2.  We will also need the automorphism groups of bicoset digraphs that can be written as $X$-joins with empty digraphs, and in Section 3 we prove all the results we will need to do this, giving their automorphism groups in Theorem \ref{xjaut}.  The last section gives the main recognition theorem (Theorem \ref{bicosjoin}), together with several corollaries and some examples which highlight some of the ways in which $X$-joins of bicoset digraphs with empty digraphs differ from the double coset digraphs which can be written as wreath products.

We expect results analogous to those in \cite{BarberDpreprint}, for example, which give the relationship between the automorphism groups of the bicoset digraphs and the Haar digraphs, to be obtained as an application of the recognition theorem (Theorem \ref{bicosjoin}).  As it is known that solving the isomorphism problem for Haar graphs depends upon properties of automorphism groups, we also expect to determine a relationship between the isomorphism problem for Haar graphs and bicoset graphs, among other results.  That work will be given elsewhere.

\section{Old and new families of digraphs}

We now define the graphs and digraphs studied in this paper, and discuss some of their symmetry properties.

\begin{definition}
Let $G$ be a group and $S \subseteq G$. Define a \textbf{Cayley digraph of $G$}, denoted $\text{Cay}(G,S)$, to be the digraph with vertex set $V(\text{Cay}(G,S)) = G$ and arc set $A(\text{Cay}(G,S)) = \{(g,gs) : g \in  \nobreak G,  s \in S\}.$ We call $S$ the \textbf{connection set of $\text{Cay}(G,S)$}.
\end{definition}

\begin{definition}\label{leftregularrepresentation}
Let $G$ be a group and $g\in G$.  Define $g_L\colon G\to G$ by $g_L(x) = gx$.  The map $g_L$ is a {\bf left translation of $G$}.  The {\bf left regular representation of $G$}, denoted $G_L$, is $G_L = \{g_L:g\in G\}$.  That is, $G_L$ is the group of all left translations of $G$. It is straightforward to verify that $G_L$ is a group isomorphic to $G$.
\end{definition}

\begin{definition} Let $G$ be a group and $H, K < G$. For each $g \in G$, the set $HgK = \{hgk : h \in H, k \in K\}$ is called the \textbf{$(H, K)$-double coset of $g$ in $G$}.  If $H = K$, then the $(H, K)$-double cosets are referred to as the \textbf{double cosets of $H$ in $G$}.
\end{definition}

Note that the $(\{1\}, K)$-double cosets are the left cosets of $K$ in $G$ and the $(H, \{1\})$-double cosets are the right cosets of $H$ in $G$. In general, $HgK$ is a union of right cosets of $H$ as well as the union of left cosets of $K$.

\begin{note}
Let $G$ be a group and $H\le G$.  We use $G/H$ to denote the set of all left cosets of $H$ in $G$, which is not necessarily a quotient group as we do not assume $H$ is normal in $G$.
\end{note}

\begin{definition}
Let $G$ be a group, $H \leq G$, and $S \subset G$ such that $S\cap H = \emptyset$ and $HSH = S$. Define a digraph $\text{Cos}(G,H,S)$ with vertex set $V(\text{Cos}(G,H,S)) = G/H$ the set of left cosets of $H$ in $G$, and arc set $A(\text{Cos}(G,H,S)) = \{(gH,gsH):g\in G{\rm\ and\ }s\in S\}$. The digraph $\text{Cos}(G,H,S)$ is called the \textbf{double coset digraph of $G$} with \textbf{connection set $S$} (or $HSH$).
\end{definition}

The condition that an ordered pair $(gH,gsH)$ is an arc of $\Cos(G,H,S)$ is often written in the equivalent form $(xH,yH)\in A(\Cos(G,H,S))$ if and only if $x^{-1}y\in S$.  Also, a Cayley digraph of $G$ is canonically isomorphic to a double coset digraph of $G$ when $H = 1$.  Similar to Cayley digraphs, for $g\in G$, the map $xH\mapsto gxH$ is an automorphism of $\Cos(G,H,S)$.  We remark that the group of all such maps has no standard name, although some people, including us, denote it by $(G,H)$.

We now give the natural bipartite analogues of Cayley and double coset digraphs.  It turns out that there are two different ways one can do this.  The oldest are the natural analogues that are graphs.

\begin{definition}
Let $G$ be a group and $S\subseteq G$.  Define the {\bf Haar graph} with {\bf connection set} $S$, denoted $\Haar(G,S)$, to be the graph with vertex set $\Z_2\times G$ and edge set $\{(0,g)(1,gs):g\in G {\rm \ and\ }s\in S\}$.
\end{definition}

We remark that some authors refer to Haar graphs as bi-Cayley graphs, and denote them accordingly.  As with Cayley graphs, note that $(0,x)(1,y)\in E(\Haar(G,S))$ if and only if $x^{-1}y\in S$.  Also note that Haar graphs are the natural bipartite analogues of Cayley digraphs and have bipartition $\{\{i\}\times G:i\in\Z_2\}$. Haar graphs were introduced in \cite{HladnikMP2002}. 

\begin{definition}
For $g\in G$, define $\widehat{g}_L : \mathbb{Z}_2 \times G \rightarrow \mathbb{Z}_2 \times G$ by $\widehat{g}_L(i,j) = (i, g_L(j))$.  It is straightforward to verify that $\widehat{g}_L$ is a bijection.   Define $\widehat{G}_L = \{\widehat{g}_L:g\in G\}$.  
\end{definition}

It is also straightforward to verify that $\widehat{G}_L$ is a group, and $\widehat{G}_L\le\Aut(\Haar(G,S))$ for every $S\subseteq G$. Also, $\widehat{G}_L$ has two orbits and its induced action on each orbit is regular and isomorphic to $G_L$.  

\begin{definition}
Let $G$ be a group, let $H_0$ and $H_1$ be subgroups of $G$, and let $S\subseteq G$ such that $S = H_0SH_1$. Define the {\bf bicoset graph} with respect to $H_0$, $H_1$, with {\bf connection set} $S$, denoted $B(G,H_0,H_1,S)$, to be the graph with vertex set  $(G/H_0) \cup (G/H_1)$ and edge set $E(\Gamma) = \{ \{gH_0,gsH_1\} : g \in G, s\in S\}$. 
\end{definition}

Bicoset graphs are bipartite with bipartition $\{\{G/H_i\}:i\in\Z_2\}$.  Bicoset graphs were introduced in \cite{DuX2000}.  Our definition is more general than in \cite{DuX2000}, where it is  assumed that the largest normal subgroup of $G$ contained in $H_0\cap H_1$ (the {\bf core} of $H_0\cap H_1$ in $G$) is trivial.  This condition ensures that the action of $G$ on $(G/H_0)\cup(G/H_1)$ is faithful.   In \cite[Lemma 2.3]{DuX2000} it was shown that the action of $G$ by left multiplication on $V(\Gamma)$ is a subgroup of $\Aut(B(G,H_{0},H_{1};S))$ with two orbits.  We denote this subgroup by $\widehat{G}$.  

In order to define bipartite digraphs that can be thought of as analogues of Cayley and double coset digraphs, we will need to have two connection sets, which give the arcs between the bipartition sets.

\begin{definition}\label{Haar defin}
Let $G$ be a group, and let $S_0,S_1 \subseteq G$. Define the \textbf{Haar digraph} of $G$ with {\bf connection sets} $S_{0},S_{1}$, denoted $\Haar(G,S_{0},S_{1})$, to be the digraph with vertex set $\mathbb{Z}_2 \times G$ and arc set $\{((0,g),(1,gs_{0})) : g \in G, \; s_{0} \in S_{0}\} \cup \{((1,g),(0,gs_{1})) : g \in G, \; s_{1} \in S_{1}\}$.
\end{definition}

Haar digraphs were first defined in \cite{DuFS2020}, and they are directed bipartite analogues of Cayley digraphs. Like Haar graphs, their automorphism groups contain $\widehat{G}_L$.  We now introduce directed bicoset graphs.

\begin{definition}
Let $G$ be a group, let $H_0$ and $H_1$ be subgroups of $G$, and let $S_{0},S_{1}\subseteq G$ such that $S_{0} = H_0S_0H_1$ and $S_{1} = H_{1} S_1 H_{0}$. Define the \textbf{ bicoset digraph} of $G$ with respect to $H_0$, $H_1$ with {\bf connection sets} $S_0$ and $S_1$, denoted $B(G,H_0,H_1;S_0,S_1)$, to be the digraph with vertex set $(G/H_0) \cup (G/H_1)$ and arc set $\{ (gH_0, gs_{0}H_1) : g \in G, s_{0} \in S_{0}\} \cup \{ (gH_{1}, gs_{1}H_{0}) : g \in G, s_{1} \in S_{1}\}$. 
\end{definition}

As usual, it is straightforward to verify that the permutation representation of the action of $G$ on $V(B(G,H_0,H_1;S_0,S_1))$ by left multiplication, denoted $\widehat{G}$, is contained in $\Aut(B(G,H_0,H_1;S_0,S_1))$, and is semiregular on $G/H_0$ and $G/H_1$. That is, $\widehat{G} = \{\widehat{g}:g\in G\}$ where $\widehat{g}:(G/H_0)\cup(G/H_1)\to (G/H_0)\cup (G/H_1)$ is defined by $\widehat{g}(kH_i) = (gk)H_i$, where $k\in G$ and $i\in\Z_2$.  It is possible that $\vert H_0\vert \not = \vert H_1\vert$, in which case the cells of the natural bipartition have different orders.

\begin{example}
Let $G = \mathbb{Z}_{15}$, $H_{0} = \la 3 \ra$, and $H_{1} = \la 5 \ra$.  One can easily see that $|B_{0}| = |G| / |H_{0}| = 3$, while $|B_{1}| = |G| / |H_{1}| = 5$. So any bicoset digraph $B(G,H_0,H_1,S_0,S_1)$ will be a bicoset digraph in which the cells of the natural bipartition have different orders, where $S_i\subseteq G$ such that $S_i = H_i S_iH_{i + 1}$, $i\in\Z_2$.  In particular, set $S_0 = \Z_{15}$ and $S_1 = \emptyset$.  Then $B(G,H_0,H_1,S_0,S_1)$ is the complete graph $K_{3,5}$ with every edge replaced with an arc from the cell of the bipartition of size $3$ to the cell of the bipartition of size $5$.
\end{example}

When $H_0 = H_1 = \{1\}$ in the definition of a bicoset digraph, the sets $S_{0}$, $S_{1}$ are just  subsets of $G$, and the bicoset digraph with respect to $H_0$, $H_1$ with connection sets $S_{0}$ and $S_{1}$ is simply $\Haar(G,S_{0},S_{1})$.  We now verify some basic facts about bicoset digraphs which are analogues of properties of double coset digraphs.  The next result is basically identical to showing that the arc set of a double coset digraph $\Cos(G,H,S)$ is well-defined if and only if its connection set $S$ is a union of $(H,H)$-double cosets (see \cite[Lemma 1.3.2]{Book}, for example).  Its proof is included for completeness. 

\begin{lemma}
Let $G$ be a group with $H_i\le G$ and $S_i\subseteq G$ for $i\in\Z_2$.  The arc set of $B(G,H_0,H_1;S_0,S_1)$ is well-defined if and only if $S_i$ is a union of $(H_{i},H_{i+1})$ double cosets, $i\in\Z_2$.
\end{lemma}

\begin{proof}
Set $\Gamma = B(G,H_0,H_1;S_0,S_1)$.  Note that for $g\in G$, $i\in\Z_2$ and $s_i\in S_i$ we have

\begin{eqnarray*}
(gH_i,gs_iH_{i+1})\in A(\Gamma) & {\rm\ if\ and\ only\ if\ } & (gh_iH_i,gh_is_iH_{i+1})\in A(\Gamma) {\rm\ for\ all\ }h_i\in H_i\\
                            & {\rm\ if\ and\ only\ if\ } & (gh_iH_i,gh_is_ih_{i+1}H_{i+1})\in A(\Gamma){\rm\ for\ all\ }h_i\in H_i, h_{i+1}\in H_{i+1}\\
                            & {\rm\ if\ and\ only\ if\ } & (gH_i,gh_is_ih_{i+1}H_{i+1})\in A(\Gamma){\rm\ for\ all\ }h_i\in H_i, h_{i+1}\in H_{i+1}\\
                            & {\rm\ if\ and\ only\ if\ } & h_is_ih_{i+1}\in S_{i} {\rm\ for\ all\ }h_i\in H_i, h_{i+1}\in H_{i+1}. 
\end{eqnarray*}
Thus $(gH_i,gs_iH_{i+1})\in A(\Gamma)$ if and only if $S_i$ is a union of $(H_i,H_{i+1})$-double cosets.      
\end{proof}

It will be useful to have some standard notation for the rest of the paper.

\begin{note}
For a bipartite digraph $\Gamma$, denote its (natural) bipartition by ${\cal B} = \{B_0,B_1\}$.  For a partition ${\cal P}$ of $V(\Gamma)$ that refines ${\cal B}$ (that is, every element of ${\cal P}$ is a subset of some element of ${\cal B}$), we denote by ${\cal P}_i$ those sets of ${\cal P}$ that are contained in $B_i$, $i\in\Z_2$.
\end{note}

\begin{definition}
Let $G\le S_n$ with orbit ${\cal O}$, and $g\in G$. Then $g$ induces a permutation on ${\cal O}$ by restricting the domain of $g$ to ${\cal O}$.  We denote the resulting permutation in $S_{\cal O}$ by $g^{\cal O}$.  The group $G^{\cal O} = \{g^{\cal O}:g\in G\}$ is the {\bf transitive constituent} of $G$ on ${\cal O}$. 
\end{definition}

Let $G$ be a group and $H\le G$.  Then $G$ acts by left multiplication on $G/H$, and this action is faithful if and only if $H$ contains no proper  nontrivial normal subgroups of $G$. 
That is, $H$ is {\bf core-free} in $G$.  Many authors (but not us) insist that when defining a double-coset digraph $\Cos(G,H,S)$, that $H$ be core-free in $G$, as if $H$ has a normal subgroup $N\tl H$, then $\Cos(G/N,H/N,T)\cong \Cos(G,H,S)$ where $T$ is the set of cosets of $N$ in $S$.  The next result shows that a similar result holds for bicoset digraphs with $\core(H_0)\cap\core(H_1)$ playing the role of $\core(H)$.  

\begin{lemma}
Let $G$ be a group, $i\in\Z_2$, $H_i \leq G$, $S_i \subseteq G$ such that $S_{i} = H_{i}S_{i}H_{i+1}$, and $\Gamma = B(G,H_{0},H_{1};S_{0},S_{1})$. Let $N = \core_G(H_{0}) \cap \core_G(H_{1})$. Then $$B(G,H_0,H_1;S_0,S_1) \cong B(G/N,H_0/N,H_1/N;\{s_{0}N : s_0 \in S_0 \},   \{s_{1}N : s_{1} \in S_{1}\}).$$
\end{lemma}

\begin{proof}
Define $\gamma:(G/H_0)\cup(G/H_1) \to ((G/N)/(H_0/N))\cup((G/N)/(H_1/N))$ by $\gamma(gH_i)=gN(H_{i}/N)$. Suppose $\gamma(gH_i) = \gamma(g'H_i)$. Then $gN(H_i/N)=g'N(H_i/N)$ or $(g^{-1}g'N)H_i/N = H_i/N$. That is, $g^{-1}g'N \in H_{i}/N$.   Thus $gH_i/N = g'H_i/N$ and $\gamma$ is a well-defined injection. As $|G/H_i| = |(G/N)/(H_i/N)|$ and $G$ is finite, we see $\gamma$ is a bijection. Let $T_i = \{s_iN : s_i \in S_i\}$. As $S_i = H_{i}S_{i}H_{i+1}$ and $N \unlhd H_{i}$, we see that $(H_i/N)T_i(H_{i+1}/N) = T_i$ as $h_{i}Ns_{i}Nh_{i+1}N = h_{i}s_{i}h_{i+1}N \in T$ for every $h_{i} \in H_{i}$, and $h_{i + 1}\in H_{i+1}$. Thus $B(G/N,H_0/N,H_1/N;T_0,T_1)$ is a well-defined digraph. Let $(gH_i,gs_iH_{i+1}) \in A(B(G,H_0,H_1;S_0,S_1))$, so $s_i \in S_i$. Then 
$$ \gamma(gH_i,gs_iH_{i+1}) = (gN(H_i/N),gs_iN(H_{i+1}/N)) \in B(G/N,H_0/N,H_1/N;T_0,T_1)$$
as $s_iN \in T_i$ as $s_i \in S_i$. The result follows.
\end{proof}

\section{$X$-joins with empty digraphs}

The role of the wreath product in \cite{BarberDpreprint} for double coset digraphs will be played by the $X$-join with a collection of empty digraphs (i.e. digraphs with no arcs) for bicoset digraphs.  This is because instead of replacing each vertex with an empty digraph of the same order for double coset digraphs, we may use empty digraphs of different orders in each cell of the bipartition ${\cal B}$ of a bicoset digraph.  In this section, we define the $X$-join, as well as prove all the results that we will need concerning the $X$-join, culminating in Corollary \ref{bicoset join} which determines the automorphism groups of all $X$-joins that we will need.

\begin{definition}\label{X join definition} Let $X$ be a digraph, and $Y = \{Y_x : x \in V(X)\}$ a collection of digraphs indexed by $V(X)$. The \textbf{$X$-join} of $Y$ is the digraph $Z = \bigvee(X, Y)$ with vertex set
$$V(Z) = \{(x,y) : x \in X, y \in Y_x \}$$
and arc set
$$A(Z) = \{ ((x,y), (x^\prime, y^\prime)) : (x, x^\prime) \in A(X) \text{ or } x = x^\prime \text{ and } (y, y^\prime) \in A(Y_x)\}.$$
\end{definition}

In other words, the digraph $Z$ is obtained by replacing each vertex of $X$ by the digraph $Y_x$ and inserting either all or none of the possible arcs between the vertices of $Y_u$ and $Y_v$ depending on whether or not there is an arc between $u$ and $v$ in $X$. If the $Y_x$'s are all isomorphic, then the $X$-join of $\{Y_x : x \in X\}$ is the wreath product $X \wr Y$, where $Y \cong Y_x$ for all $x \in X$.

Much of the work on automorphism groups of wreath products and $X$-joins has been based on two equivalence relations, one of which ($R_\Gamma$) we now define. 

\begin{definition}
Let $\Gamma$ be a digraph. Define the equivalence relations $R_{\Gamma}$ and $D_{\Gamma}$ 
on $V(\Gamma)$ as follows:
\begin{enumerate}
\item $u \ R_{\Gamma} \ v$ if and only if $N^{+}_\Gamma(u) =  N^{+}_\Gamma(v)$ and $N^{-}_\Gamma(u) = N^{-}_{\Gamma}(v)$, and
\item $u \ D_\Gamma \ v$ if and only if $u=v$.
\end{enumerate}
$\Gamma$ is called \textbf{irreducible} if $R_{\Gamma} = D_{\Gamma}$, otherwise it is called \textbf{reducible}. We call the set of equivalence classes of $R_\Gamma$ the {\bf unworthy} partition of $V(\Gamma)$.  $D_{\Gamma}$ is called the diagonal (or identity) relation. We will abbreviate the condition in (1) by saying $v \ R_{\Gamma} \ u$ if $N^{\pm}_\Gamma(v) = N^{\pm}_\Gamma(u)$.
\end{definition}

We note that $R_{\Gamma}$ is a $G$-congruence (see \cite[Definition 2.3.1]{Book}) if $G \leq \Aut(\Gamma)$ is transitive.  Some people say $u$ and $v$ are {\bf twins} if $u\ R_{\Gamma}\ v$.

Our terminology here is taken from a wide variety of sources, as the equivalence relation $R_\Gamma$ has been discovered many times by many people in many contexts.  The term irreducible is due to Sabidussi \cite{Sabidussi1964} and unworthy is due to Wilson \cite{Wilson2003}.  We will need the following result \cite[Lemma 1(ii)]{Sabidussi1964}.

\begin{lemma}\label{Gamma mod is irreducible}
If $\Gamma$ is a reducible digraph with unworthy partition ${\cal P}$, then $\Gamma/{\cal P}$ is irreducible.
\end{lemma}

Bipartite digraphs which can be written as $X$-joins are easy to recognize.

\begin{definition}
Let $X$ be a digraph and $Y$ a collection of digraphs indexed by $V(X)$.  The partition $\{\{(x,y):y\in V(Y_x)\}:x\in V(X)\}$ is the {\bf join partition of $\bigvee(X,Y)$}.
\end{definition}

Hemminger \cite[Theorem 2.8]{Hemminger1968} gave necessary and sufficient conditions for the automorphism group of an $X$-join of $Y$ to consist of what he calls the ``natural" automorphisms. 

\begin{definition} Let $Z$ be an $X$-join of $\{Y_x\}_{x \in X}.$  Then a digraph automorphism $\varphi$ of $Z$ is called \textbf{natural} if for each $x_1 \in X$ there is an $x_2 \in X$ such that $\varphi(\{(x_1,y):y\in V(Y_{x_1})\}) = \{(x_2,y):y\in V(Y_{x_2})\}$.  Otherwise, it is \textbf{unnatural}.
\end{definition}

So a natural automorphism of $Z$ is an automorphism that permutes the elements of $\{Y_x\}_{x\in X}$, while an unnatural element is one which maps 
vertices in one $Y_x\in\{Y_x\}_{x\in X}$ to vertices in at least two different elements of $\{Y_x\}_{x\in X}$.  It is important to note that if $\varphi$ is a natural automorphism of $Z$, then $\varphi$ induces an automorphism $\varphi^*$ of $X$ where $\varphi^*(x_1) = x_2$ if $\varphi(\{(x_1,y):y\in V(Y_{x_1})\}) = \{(x_2,y):y\in V(Y_{x_2})\}$.  Similarly, if $\phi$ is an automorphism of $X$ such that $Y_{x} \cong Y_{\phi(x)}$ for all $x \in X$, then $\phi$ induces a set of natural automorphisms of $Z$, where if $\phi^\prime$ is one of these natural automorphisms, then $(\phi^\prime)^* = \phi$.

\begin{definition}
Let $X$ be a digraph, $Y$ a collection of digraphs indexed by $V(X)$, and $Z = \bigvee(X,Y)$.  We say that the set of natural automorphisms is {\bf complete} if for each automorphism $\sigma$ of 
$X$ there exists a natural automorphism $\mu$ of $Z$ such that $\mu^* = \sigma$.
\end{definition}

The next result is well known, and shows that if $R_\Gamma\not = D_\Gamma$, then $\Aut(\Gamma)$ has automorphisms of a particular form.  The proof is provided here for completeness.

\begin{lemma}\label{point of R}
Let $\Gamma$ be a digraph and $E$ be an equivalence class of $R_\Gamma$.  Let $\sigma\in S_E$.  Define a permutation $\bar{\sigma}\in S_{V(\Gamma)}$ by $\bar{\sigma}(v) = \sigma(v)$ if $v\in E$ and $\bar{\sigma}(v) = v$ otherwise.  Then $\bar{\sigma}\in\Aut(\Gamma)$.  
\end{lemma}

\begin{proof}
Since for $v,v'\in E$ the neighbors of $v$ and $v'$ are the same, we see that $\bar{\sigma}$ preserves the arcs of $\Gamma$ incident with a vertex of $E$.  As $\bar{\sigma}^{V(\Gamma) - E} = 1$, $\bar{\sigma}$ preserves the arcs of $\Gamma$ that are not incident with a vertex of $E$.  Thus $\bar{\sigma}$ preserves the arcs of $\Gamma$, and so $\bar{\sigma}\in \Aut(\Gamma)$.
\end{proof}

\begin{theorem}\label{bipartite join}
Let $X$ be a digraph and $Y = \{Y_x:x\in V(X)\}$ where each $Y_x\in Y$ is an empty digraph.   If $\Gamma\cong\bigvee(X,Y)$, then the following are equivalent:
\begin{enumerate}
\item\label{Gen 1} $X$ is irreducible, 
\item\label{Gen 2} the join partition of $\Gamma$ is the unworthy partition of $\Gamma$,
\item\label{Gen 3} every automorphism of $\Gamma$ is natural.
\end{enumerate}
\end{theorem}

\begin{proof}
(\ref{Gen 1}) $\implies$ (\ref{Gen 2}): Suppose $X$ is irreducible. As $Y_{x}$ is empty for all $x \in V(X)$, $u \ R_{\Gamma} \ v$ for all $u,v \in \{ x \} \times V(Y_{x})$. This implies that the unworthy partition of $\Gamma$ is refined by the join partition of $\Gamma$. As $X$ is irreducible, for any pair of distinct vertices $x_{1},x_{2} \in V(X)$, there exists a vertex $x_{3} \in V(X)$ such that $(x_{1},x_{3}) \in A(X)$ but $(x_{2},x_{3}) \not \in A(X)$ or such that $(x_{3},x_{1}) \in A(X)$ but $(x_{3},x_{2}) \not \in A(X)$.

As there exists a vertex $x_{3} \in V(X)$ such that $x_{1}$ is in-adjacent (resp. out-adjacent) to $x_{3}$ but $x_{2}$ is not in-adjacent (resp. out-adjacent) to $x_{3}$, no vertex of $\{(x_{2},y) : y \in V(Y_{x_{2}})\}$ is in-adjacent (resp. out-adjacent) to any vertex of $\{(x_{3},y) : y \in V(Y_{x_{3}})\}$, while every vertex of $\{(x_{1},y) : y \in V(Y_{x_{1}})\}$ is in-adjacent (resp. out-adjacent) to every vertex of $\{(x_{3},y) : y \in V(Y_{x_{3}})\}$. Thus, the cell of $R_{\Gamma}$ containing $\{(x_{1},y) : y \in V(Y_{x_{1}})\}$ must contain every element of the form $\{(x_{1},y) : y \in V(Y_{x_{1}})\}$ and as $x_{1},x_{2}$ are arbitrary, the cell of $R_{\Gamma}$ containing $\{(x_{1},y) : y \in V(Y_{x_{1}})\}$ does not contain any element of the form $\{(x_{2},y) : y \in V(Y_{x_{i}})\}$. As $x_1$ and $x_2$ are arbitrary, the cell of $R_{\Gamma}$ containing $\{(x_{1},y) : y \in V(Y_{x_{1}})\}$ is $\{(x_{1},y) : y \in V(Y_{x_{1}})\}$.

(\ref{Gen 2}) $\implies$ (\ref{Gen 3}):  Suppose that the join partition of $\Gamma$ is the unworthy partition of $\Gamma$. Let $\gamma \in \Aut(\Gamma)$ and $u,v\in V(\Gamma)$ such that $u \ R_{\Gamma} \ v$. Then $\gamma(u) \ R_{\Gamma} \ \gamma(v)$. Thus $\gamma$ maps the unworthy partition of $\Gamma$ to the unworthy partition of $\Gamma$, and so maps the join partition of $\Gamma$ to the join partition of $\Gamma$. This implies that $\gamma$ is a natural automorphism.

(\ref{Gen 3}) $\implies$ (\ref{Gen 1}): Suppose that every automorphism of $\Gamma$ is natural.  Towards a contradiction, suppose that $X$ is reducible. Then, there exist vertices $x_{1},x_{2} \in V(X)$ such that $N_{X}^{\pm}(x_{1}) = N_{X}^{\pm}(x_{2})$. This implies, by definition, that every vertex of $\{(x_{1},y) : y \in V(Y_{x_{1}}) \}$ and every vertex of $\{(x_{2},y) : y \in V(Y_{x_{2}}) \}$ has the same out- and in-neighbors in $\Gamma$. The permutation $((x_{1},y_{1}),(x_{2},y_{2}))$, for all $y_{1} \in V(Y_{x_{1}})$ and $y_{2} \in V(Y_{x_{2}})$ is an automorphism of $\Gamma$ and is an unnatural automorphism, a contradiction. 
\end{proof}

The terminology and notation given above was introduced by Hemminger \cite{Hemminger1968}, where he studied the automorphism group of $Z = \bigvee(X,Y)$. In general, even if every automorphism of $\Aut(Z)$ is natural, $\Aut(Z)$ can be difficult to describe.  In fact, Hemminger gave necessary and sufficient conditions for every automorphism of $Z$ to be natural \cite[Theorem 2.8]{Hemminger1968}, but did not write an explicit description of what this group looks like (as it depends upon which of the $Y_x$'s are isomorphic).  However, in the introduction, he gave the following description of what a natural automorphism is (this is taken more or less verbatim from \cite{Hemminger1968}, with the only differences in notation):

\bigskip

\noindent If $\mu\in\Aut(X)$, such that $Y_x \cong  Y_w$ whenever $\mu(x) = w$ then $\mu$ induces a set of ``natural" graph automorphisms of $Z$, namely, the automorphisms that are obtained by first permuting the $Y_x$ depending on how $\mu$ permutes the subscripts of the $Y_x$ and then performing an arbitrary automorphism of each $Y_x$.

\bigskip

In the case of bicoset digraphs $Z$, there are only two possibilities for the set $Y$ of empty graphs (this is explicitly proven below).  Either $\Aut(X)$ is transitive and all the empty graphs in $Y$ have the same number of elements (and consequently $\Aut(Z) = \Aut(X)\wr S_m$, where $m = \vert Y_x\vert$, $Y_x\in Y$), or $\Aut(X)$ has two orbits, and if $X$ is weakly connected, any two elements of $Y$ contained in $B_{i} \in{\cal B}$, $i \in \Z_2$ have the same order.  Now, a complete set of natural automorphisms need not be a group.  However, it is easy to see that the set of all natural automorphisms of $Z$ is a group.  

\begin{definition}
Let $X$ be a digraph, $Y$ a set of digraphs indexed by $X$, and $Z = \bigvee(X,Y)$.  We will call the set of all natural automorphisms of $Z$ the {\bf group of natural automorphisms of $Z$}.  Additionally, if the group of natural automorphisms is also a complete set, we will call it the {\bf complete group of natural automorphisms}.
\end{definition}

For a non-vertex-transitive bicoset digraph, the group of natural automorphisms will be generated by elements of $\Aut(X)$ permuting the elements of $Y$ whose orders are the same, as such an element permutes the subscripts, together with the subgroups that are the symmetric groups on $Y_x$ and the identity elsewhere, for every $Y_x\in Y$.  The following result then gives the automorphism group of bicoset digraphs that can be written as a join that we will need in this paper.  We will need some common permutation group theoretic terms first.

\begin{definition}\label{imprimitive}
Let $X$ be a set and $G\le S_X$ be transitive.  A subset $B\subseteq X$ is a {\bf block} of $G$ if whenever $g\in G$, then $g(B)\cap B = \emptyset$ or $B$.  
\end{definition}

So a block $B$ of a permutation group $G\le S_X$ is a subset of $X$ which is either fixed set-wise by $g\in G$, or is moved ``away" from $B$.  Note that if $B$ is a block of $G$, then $g(B)$ is also a block of $B$ for every $g\in G$, and is called a {\bf conjugate block of $B$}.

\begin{definition}
The set of all blocks conjugate to $B$ is a partition of $X$, called a {\bf block system of $G$}. 
\end{definition}

\begin{corollary}\label{bicoset join}
Let $X$ be an irreducible bicoset digraph that is weakly connected when $\Aut(X)$ is transitive and $Y = \{Y_x:x\in V(X)\}$, where each $Y_x\in Y$ is an empty digraph.   Let $\Gamma\cong\bigvee(X,Y)$.  The following are equivalent:
\begin{enumerate}
\item\label{Case 1} $\Aut(\Gamma)$ is the complete group of natural automorphisms, 
\item\label{Case 2} either $X$ is vertex-transitive and every element of $Y$ has the same order, or any two elements of $Y$ contained in $B_{i}\in{\cal B}$, $i \in \Z_2$, have the same order.
\end{enumerate}
\end{corollary}

\begin{proof}
As $X$ is weakly connected when  $\Aut(X)$ is transitive, $\Gamma$ is weakly connected, and so ${\cal B}$ is a block system of $\Aut(\Gamma)$ when $\Gamma$ is vertex-transitive \cite[Lemma 1.10]{Dobson2022}, while ${\cal B}$ is the set of orbits of $\Aut(\Gamma)$ otherwise. As $X$ is irreducible, by Theorem \ref{bipartite join} we see that every automorphism of $\Gamma$ is natural. 

Suppose that $\Aut(\Gamma)$ is the complete group of natural automorphisms.  If $X$ is vertex-transitive, then for every $x_1,x_2\in V(X)$ there is $\sigma\in \Aut(X)$ (which depends on $x_1,x_2$) such that $\sigma(x_1) = x_2$.  Then there exists $\delta\in \Aut(\Gamma)$ such that $\delta^* = \sigma$.  Hence $\delta(Y_{x_1}) = Y_{x_2}$.  Thus $\vert Y_{x_1}\vert = \vert Y_{x_2}\vert$.  As $x_1$ and $x_2$ are arbitrary, every element of $Y$ has the same order.  If $X$ is not vertex-transitive, then $\Aut(X)$ is transitive on $B_{i} \in{\cal B}$, and any two elements of $Y$ contained in $B_{i}\in{\cal B}$ have the same order follows by an analogous argument.

Suppose that $X$ is vertex-transitive and every element of $Y$ has the same order.  Let $T$ be a set of size $\vert Y_x\vert$, $Y_x\in Y$, and assume without loss of generality that for every $x\in X$, we have $Y_x = \{(x,t):t\in T\}$.  Then $V(\Gamma) = V(X)\times T$.  Let $\sigma\in\Aut(X)$. Define $\delta:V(X)\times T\to V(X)\times T$ by $\delta(x,s) = (\sigma(x),\omega_x(s))$, where $\omega_x\in S_T$ for every $x\in V(X)$.  Let $((x_1,t_1),(x_2,t_2))\in A(\Gamma)$.  Then $(x_1,x_2)\in A(X)$, and so $\sigma(x_1,x_2)\in A(X)$.  Hence $((\sigma(x_1),t_3),(\sigma(x_2),t_4))\in A(\Gamma)$ for every $t_3,t_4\in T$, and $\delta((x_1,t_1),(x_2,t_2))\in A(\Gamma)$. Thus $\delta\in\Aut(\Gamma)$.  The result follows in this case.  The argument for when $X$ is not vertex-transitive but $X$ is a bicoset digraph is analogous, although instead of having one set $T$, we will need two sets $T_1$ and $T_2$, one for each $B_{i} \in{\cal B}$. 
\end{proof}

\section{Automorphism groups of reducible bicoset digraphs}

We will have need of quotient digraphs, but as bicoset digraphs need not be vertex-transitive \cite{DuX2000}, the natural partition by which to quotient them has two parts (one for each cell of the natural bipartition), and so is slightly more complicated. We first give this partition, which is a special case of Definition \ref{join partition}.

\begin{definition}
Let $G$ be a group, $i \in \mathbb{Z}_{2}$, $H_{i} \leq K_{i} \leq G$, and $S_{i} \subseteq G$ such that $S_{i} = H_{i}S_{i}H_{i+1}$. Let $\Gamma = B(G, H_0, H_1; S_{0}, S_{1})$, and let $\mathcal{P}_{i}$ be the partition of $B_{i}$ where each cell is the set of left cosets of $H_{i}$ in $G$ contained in a left coset of $K_{i}$ in $G$.  Define the \textbf{join partition of $V(\Gamma)$ with respect to $H_i$ and $K_i$}, denoted $\mathcal{P}(H_i,K_i)$, of the vertices of $\Gamma$ as 
$\mathcal{P}(H_{i},K_{i}) = \mathcal{P}_{0} \cup \mathcal{P}_{1}$. That is, $\mathcal{P}_{i} = \{\{k_{i}H_{i}:k_{i} \in gK_i\} : g \in G, i\in\Z_2\}$.
\end{definition}

Note that $\mathcal{P}_{i}$ is a block system of $\widehat{G}^{B_{i}}$, as $H_{i}$ partitions $G$ and $H_{i} \leq K_{i}$, but the join partition $\mathcal{P}(H_i,K_i)$ of $V(\Gamma)$ is not necessarily a block system of $\Aut(\Gamma)$, as $\Gamma$ may not be vertex-transitive.  For example, if $\vert K_0/H_0\vert\not = \vert K_1/H_1\vert$ then $\Gamma$ is never vertex-transitive.  

\begin{definition}
Let $X$ be a set, and $G\le S_X$ (we note that $G$ need not be transitive).  A partition ${\cal P}$ of $X$ will be called a {\bf $G$-invariant partition of $X$} if $g(P)\in{\cal P}$ for every $P\in{\cal P}$.
\end{definition}

We observe that if $G$ is transitive, a $G$-invariant partition is simply a block system of $G$.  However, we will only use this terminology when the group $G$ is intransitive.  The next result shows that for us, the join partition ${\cal P}(H_i,K_i)$ is natural in our context.

\begin{lemma}\label{join partition}
Let $G$ be a group and $i \in \mathbb{Z}_{2}$.  Let ${\cal P}$ be a partition of $\Z_2\times G$ that refines ${\cal B}$.  Then ${\cal P}$ is a $\wh{G}$-invariant partition of $\Z_2\times G$ if and only if there exist $H_i\le K_i\le G$, such that ${\cal P} = {\cal P}(H_i,K_i)$.
\end{lemma}

\begin{proof}
It is clear that ${\cal P}(H_i,K_i)$ is a refinement of ${\cal B}$ and invariant under $\widehat{G}$ as left multiplication permutes left cosets of any subgroup of $G$, where $H_i\le K_i\le G$.  Conversely, suppose that ${\cal P}$ refines ${\cal B}$ and is $\widehat{G}$-invariant. Let ${\cal P}_i$ consist of those subsets of ${\cal P}$ that are contained in $B_i$.  As ${\cal P}$ refines ${\cal B}$, ${\cal P}_i$ is a partition of $B_i$.  Furthermore, $\widehat{G}^{B_i}$ is transitive on $B_i$, so ${\cal P}_i$ is a block system of $\widehat{G}^{B_i}$.  By \cite[Exercise 2.3.16]{Book} ${\cal P}_i$ is the set of left cosets of some $H_i\le K_i\le G$.  So ${\cal P} = {\cal P}(H_i,K_i)$.
\end{proof}

\begin{lemma}\label{unworthy partition G-invariant}
Let $G$ be a group, $H_{i} \leq G$, and $S_{i} \subseteq G$ such that $S_{i} = H_{i}S_{i}H_{i+1}$.  Let $\Gamma = B(G,H_0,H_1;S_0,S_1)$.  The unworthy partition of $\Gamma$ is $\widehat{G}$-invariant.
\end{lemma}

\begin{proof}
Let $u,v  \in V(\Gamma)$, such that $N^{\pm}_\Gamma(u) = N^{\pm}_\Gamma(v)$. For any in- or out-neighbor $y$ of $u$ and $v$ and $g\in G$, we see that $\widehat{g}(y)$ is an in- or out-neighbor of $\widehat{g}(u)$ and $\widehat{g}(v)$. Thus $N^{\pm}_\Gamma \left( \wh{g}(u) \right) = N^{\pm}_{\Gamma} \left( \wh{g}(v) \right)$. 
\end{proof}

\begin{lemma}\label{P refines B}
Let $\Gamma$ be a bipartite digraph without isolated vertices with bipartition ${\cal B} = \{B_0,B_1\}$.  The unworthy partition ${\cal P}$ of $\Gamma$ refines ${\cal B}$.
\end{lemma}

\begin{proof}
Suppose towards a contradiction that there exists some $P \in \mathcal{P}$ such that $P \cap B_{0} \neq \emptyset \not= P\cap B_1$. Let $x_{i} \in V(\Gamma)$ be such that $x_{i} \in P \cap B_{i}$, $i\in\Z_2$. As $\Gamma$ has no isolated vertices, $x_{0}$ has some in- or out-neighbor $y$ which, as $\Gamma$ is bipartite, is contained in $B_{1}$. But, as the in- and out-neighborhoods of $x_{0}$ and $x_{1}$ are the same, either $(x_{1},y) \in A(\Gamma)$ or $(y,x_{1}) \in A(\Gamma)$. In either case, this implies that $\mathcal{B}$ is not a bipartition of $\Gamma$, a contradiction.
\end{proof}

\begin{lemma}\label{R and Kernel}
Let $G$ be a group, $i \in \mathbb{Z}_{2}$, $H_{i} \leq G$, and $S_{i} \subseteq G$ such that $S_{i} = H_{i}S_{i}H_{i+1}$.  Let $\Gamma = B(G,H_0, H_1; S_{0}, S_{1})$.  Let ${\cal P}$ be the unworthy partition of $V(\Gamma)$, and ${\cal P}_i$ consist of those elements of ${\cal P}$ that are contained in $B_i$.  Then ${\cal P}_i$ is the set of orbits of the kernel of the induced action of $F$ on $B_{i+1}$, where $F$ is the set-wise stabilizer of each $B_{i} \in{\cal B}$, and ${\cal P} = {\cal P}(H_i,K_i)$ for some $H_i\le K_i\le G$.
\end{lemma}

\begin{proof}
By Lemma \ref{P refines B}, we see that $\mathcal{P}_{i}$ is a partition of $B_{i}$, $i\in\Z_2$. Let ${\rm{Ker}}_{i}$ be the kernel of the action of $F$ on $B_{i+1}$ and $\mathcal{O}$ be an orbit of ${\rm{Ker}}_{i}$.  Fix $v \in \mathcal{O}$ and let $P_{i} \in \mathcal{P}_{i}$ such that $v \in P_{i}$. By definition, ${\rm{Ker}}_{i}$ fixes every vertex of $B_{i+1}$. Note that for every $y \in N^{\pm}(v)$, $y \in B_{i+1}$ and ${\rm{Ker}}_{i}^{\mathcal{O}}$ acts transitively on $\mathcal{O}$. Thus if $x \in N^{+}(v)$, then $u \in N^{-}(x)$ for all $u \in \mathcal{O}$, and similarly, if $x \in N^{-}(v)$ then $u \in N^{+}(x)$ for all $u \in \mathcal{O}$. We conclude that if $x$ is an in- and out-neighbor of $v$, then $x$ is an in- and out-neighbor of every element of $\mathcal{O}$. Thus every element of $\mathcal{O}$ has the same neighborhood, and $\mathcal{O} \subseteq P_{i}$. Conversely, let $P_{i} \in \mathcal{P}_{i}$ and for each $\sigma \in S_{P_{i}}$, define $\overline{\sigma} \in S_{V(\Gamma)}$ by $\overline{\sigma}(v) = \sigma(v)$ if $v \in P_{i}$ and $\overline{\sigma}(v)=v$ otherwise. By Lemma \ref{point of R} $\overline{\sigma} \in F$, and so $P_{i} \subseteq \mathcal{O}$. Thus $P_{i} = \mathcal{O}$ and so $\mathcal{P}_{i}$ is the set of orbits of the kernel of the action of $F$ on $B_{i+1}$. As by Lemma \ref{unworthy partition G-invariant} $\mathcal{P}$ is $\widehat{G}$-invariant, by Lemma \ref{join partition}, there exist $K_i$ satisfying $H_i \leq K_i \leq G$, and $\mathcal{P}_{i} = \mathcal{P}(H_{i},K_{i})$.  
\end{proof}

We now give the automorphism group of every bicoset digraph that can be written as an $X$-join of a set of empty digraphs $Y$ such that the partition ${\cal P} = \{V(Y_x):x\in X\}$ refines the natural bipartition ${\cal B}$. But first, with the natural partition in hand, we now define quotient digraphs.  

\begin{definition}
Let $\Omega$ be a set and $\mP$ be a partition of $\Omega$. Let $\Gamma$ be a digraph with vertex set $\Omega$. Define the \textbf{quotient digraph} of $\Gamma$ with respect to $\mP$, denoted $\Gamma/\mP$, by $V(\Gamma/\mP) = \mP$ and $(P_1, P_2) \in A(\Gamma/\mP)$ if and only if $(p_1, p_2) \in A(\Gamma)$ for some $p_1 \in P_1$ and $p_2 \in P_2.$
\end{definition}

\begin{theorem}\label{xjaut}
Let $G$ be a group, $i \in \mathbb{Z}_{2}$, $H_{i} \leq G$, and $S_{i} \subseteq G$ such that $S_{i} = H_{i}S_{i}H_{i+1}$.  Let $\Gamma = B(G,H_0,H_1;S_0,S_1)$ be a bicoset digraph such that $\Gamma$ is weakly connected when $\Aut(\Gamma)$ is transitive and $\Gamma = \bigvee(X,Y)$ for some digraph $X$ and $Y$ a collection of empty graphs that refines the natural bipartition ${\cal B}$ of $\Gamma$.  Then there exists $H_i\le K_i\le G$ such that the unworthy partition ${\cal P}$ of $V(\Gamma)$ is the join partition ${\cal P}(H_i,K_i)$, and if $Y'$ is the collection of empty graphs on the cells of ${\cal P}$, then $\Gamma\cong \bigvee(\Gamma/{\cal P},Y')$. In addition, $\Aut(\bigvee(\Gamma/{\cal P},Y'))$ is the complete group of natural automorphisms of $\bigvee(\Gamma/{\cal P},Y')$. Also, if every empty digraph in $Y'$ has the same order $n$, then $\Aut(B(G,H_0,H_1;S_0,S_1)) = \Aut(\Gamma/{\cal P})\wr S_n$.
\end{theorem}

\begin{proof}
Let $F$ be the set-wise stabilizer of each block of $B_{i} \in{\cal B}$, and ${\cal P}_i$ consist of the elements of ${\cal P}$ contained in $B_i$. By Lemma \ref{R and Kernel}, ${\cal P}_i$ is the set of orbits of the kernel of the induced action of $F$ on $B_{i+1}$, and ${\cal P} = {\cal P}(H_i,K_i)$ for some $H_i\le K_i\le G$. Thus the elements of ${\cal P}_i$ have the same order, and if $\Gamma$ is vertex-transitive, all the elements of ${\cal P}$ have the same order.  Let $L$ be the subgroup of $\Aut(\Gamma/{\cal P})$ that fixes each cell of the bipartition of $\Gamma/{\cal P}$, call it ${\cal B}/{\cal P}$, set-wise.  

Suppose towards a contradiction that the induced action of $L$ on the left cosets of $K_{i + 1}$ in $B_{i+1}$ is unfaithful.  Then the kernel of the action fixes each left coset of $K_{i+1}$ in $B_{i+1}$, and is nontrivial on the left cosets of $K_i$ in $B_i$.  Then there exists $\delta_i\in\Aut(\Gamma/{\cal P})$ that is trivial on the left cosets of $K_{i+1}$ in $B_{i+1}$ but $\delta_{i}(K_{i}) = g_{i}K_{i} \neq K_{i}$ for some $g_{i} \in G$. Then $K_{i}$ and $g_{i}K_{i}$ have the same neighborhood in the quotient $\Gamma / \mathcal{P}$.  If $\Gamma$ has no arcs, then $\Aut(\Gamma)$ is a symmetric group and transitive.  By hypothesis, $\Gamma$ is weakly connected, a contradiction.  Thus $\Gamma$ has arcs.  Then some left coset $k_iH_i$ of $K_i$ is in- or out-adjacent to some left coset $\ell_{i+1}H_{i+1}$ in $\ell_{i+1}K_{i+1}$, while some left coset $g_iH_i$ in $g_iK_i$ is in- or out-adjacent to $\ell_{i+1}'H_{i+1}$ that is also contained in $\ell_{i+1}K_{i+1}$.  As ${\cal P}$ is the unworthy partition of $\Gamma$, this implies that every left coset of $H_i$ in $K_i$ and $g_iK_i$ is in- and out-adjacent to every left coset of $H_{i+1}$ in $\ell_{i+1}K_{i+1}$.  As $K_i$ and $g_iK_i$ have the same neighbors in $\Gamma/{\cal P}$, this implies that $H_{i}$ and $g_{i}H_{i}$ have the same neighbors in $\Gamma$, and so $H_{i} \ R_{\Gamma} \ g_{i}H_{i}$. Then $K_{i}$ contains $g_{i}K_{i}$, a contradiction. 

Thus the action of $L$ on ${\cal B}/{\cal P}$ is faithful. So, $R_{\Gamma / \mathcal{P}} = D_{\Gamma / \mathcal{P}}$, and $\Gamma/{\cal P}$ is irreducible. As if $\Gamma$ is weakly connected when $\Aut(\Gamma)$ is transitive, by Corollary \ref{bicoset join}, $\Aut(\Gamma) = \Aut(\bigvee(\Gamma / \mathcal{P},Y'))$ is the complete group of natural automorphisms of $\bigvee(\Gamma / \mathcal{P},Y')$. Additionally, if every empty digraph in $Y'$ has the same order $n$, then $\Gamma \cong (\Gamma / \mathcal{P}) \wr \overline{K}_{n}$ and it follows that $\Aut(B(G,H_{0},H_{1};S_{0},S_{1})) = \Aut(\Gamma / \mathcal{P}) \wr S_{n}$.
\end{proof}

\begin{corollary}
Let $G$ be a group, $i \in \mathbb{Z}_{2}$, and $S_{i} \subseteq G$.  Let $\Gamma = \Haar(G,S_0,S_1)$ be a Haar digraph such that $\Gamma$ is weakly connected when $\Aut(\Gamma)$ is transitive, and $\Gamma = \bigvee(X,Y)$ for some digraph $X$ and $Y$ a collection of empty graphs that refines the natural bipartition ${\cal B}$ of $\Gamma$. Then there exists $K_i \leq G$, such that the unworthy partition ${\cal P}$ of $V(\Gamma)$ is the join partition ${\cal P}(1, K_i)$ and if $Y'$ is the collection of empty digraphs on the cells of $\mP$, then $\Gamma \cong \bigvee(\Gamma/\mP, Y')$.  Also, $\Aut(\bigvee(\Gamma/{\cal P},Y'))$ is the complete group of natural automorphisms of $\bigvee(\Gamma/{\cal P},Y')$. Additionally, if every empty digraph in $Y'$ has the same order $n$, then $\Aut(\Haar(G,S)) = \Aut(\Gamma/{\cal P})\wr S_n$.
\end{corollary}

\begin{proof}
When $H_0 = H_1 = \{1_G\}$, $\Gamma = B(G,H_0,H_1;S_0,S_1) = \text{Haar}(G,S_0,S_1)$. The result follows from Corollary \ref{xjaut}.
\end{proof}

\section{Recognition of bicoset digraphs which are $X$-joins}

Originally, our main goal for this section was to determine when a Haar graph is isomorphic to a wreath product of a graph and an empty graph by inspection of its connection set.  However, just as quotients of Cayley digraphs need not be Cayley digraphs, quotients of Haar digraphs need not be Haar digraphs.  So the proper context in which to proceed is in the class of digraphs which contains all bipartite digraphs and has a group acting transitively on each of the bipartition classes, as Haar digraphs have this property.  These digraphs are the bicoset digraphs, as previously stated. 

In the previous section, we determined the automorphism group of a bicoset digraph when it is isomorphic to a non-trivial $X$-join. In this section, we determine when a bicoset digraph is isomorphic to a non-trivial $X$-join by inspecting the connection sets $S_{0}$ and $S_{1}$.

We first prove a variation of a recognition theorem for wreath products (see, for example, \cite[Theorem 4.2.15]{Book}) adapted for bicoset digraphs.  We remind the reader that bicoset digraphs need not be vertex-transitive, and we do not assume that they are here.  We will first need a definition.  

\begin{definition}
A group $G\le S_n$ is {\bf semitransitive} if $G$ has exactly two orbits.
\end{definition}

In \cite[Lemma 2.4]{DuX2000} it was shown that every graph whose automorphism group has a semitransitive subgroup is isomorphic to a bicoset graph.  A similar proof also works for digraphs.  The next result generalizes \cite[Theorem 4.2.15]{Book} from wreath products to $X$ joins, and can be thought of as a digraph theoretic version of the recognition theorem.

\begin{theorem}
Let $\Gamma$ be a bipartite digraph whose bipartition ${\cal C} = \{C_i:i\in\Z_2\}$ is the orbits of a semitransitive group $G$.  Let $i\in\Z_2$.  Let ${\cal P}\preceq{\cal C}$ be a $G$-invariant partition, and ${\cal P}_i$ be the sets of ${\cal P}$ contained in $C_i$.  Let $X = \Gamma/{\cal P}$, and  $Y$ be the set of all empty digraphs whose vertex-sets are elements of ${\cal P}$.  Then $\Gamma = \bigvee(X,Y)$ if and only if whenever $P_i \in \mathcal{P}_i$ and $P_{i+1} \in \mathcal{P}_{i+1}$, then there is an arc $(x_i,x_{i+1})$ from a vertex $x_i \in P_i$ to a vertex $x_{i+1} \in P_{i+1}$ if and only if every arc of the form $(x_i, x_{i+1})$ with $x_i \in P_i$ and $x_{i+1} \in P_{i+1}$ is contained in $A(\Gamma)$.
\end{theorem}

\begin{proof}
Set ${\cal P} = {\cal P}(H_i,K_i)$.  Suppose $\Gamma = \bigvee(X,Y)$. Let $x_{i} \in P_{i} \in \mathcal{P}_{i}$, $i\in\Z_2$. Then
$$\begin{array}{lcl}
(x_{i},x_{i+1}) \in A(\Gamma) & \text{if and only if} & (P_{i}, P_{i+1}) \in A(\Gamma/\mP)\\
				& \text{if and only if} & (x_{i},x_{i+1}) \in A(\Gamma) 
                  \text{ for all } x_{i} \in P_{i} \text{ and } x_{i+1} \in P_{i+1}
\end{array}$$
by the definition of the $X$-join of $Y$.

For the converse, first observe that $V(\bigvee(X,Y)) = V(\Gamma)$.  Let $i\in\Z_2$, $x\in P_i$ and $y\in P_{i+1}$.  As ${\cal P}\prec{\cal C}$ and ${\cal C}$ is a bipartition of $\Gamma$, there is $x\in P_x\in{\cal P}_i$ and $y\in P_y\in{\cal P}_{i+1}$. Suppose $(x,y)\in A(\Gamma)$.  Then $(P_x,P_y)\in A(\Gamma/{\cal P})$ and $(x,y)\in A(\bigvee(X,Y))$.  Thus $A(\Gamma)\subseteq A(\bigvee(X,Y))$.  Now suppose $(x,y)\not\in A(\Gamma)$.  Then there is no arc from any vertex of $P_i$ to any vertex of $P_{i + 1}$, and so $(P_x,P_y)\not\in A(\bigvee(X,Y))$.  Thus $A(\Gamma) = A(\bigvee(X,Y))$, and the result follows.
\end{proof}

The next result is the version of the previous result for bicoset digraphs.

\begin{theorem} \label{bicoswr}
Let $G$ be a group, $i \in \Z_2$, $H_i \leq K_i\le  G$, and $S_i \subseteq G$ such that $S_i = H_{i}S_{i}H_{i+1}$. Let $\Gamma = B(G,H_0,H_1;S_{0},S_{1})$, and let $X = \Gamma/\mathcal{P}(H_i,K_i)$.  For $g\in G$, let $Y_{i,g}$ be the empty digraph on the left cosets of $H_i$ contained in $gK_i$, and $Y = \{Y_{i,g}:i\in\Z_2,g\in G\}$. Let ${\cal P}_i$ be the sets in ${\cal P}(H_i,K_i)$ contained in $B_i$.  Then $\Gamma \cong \bigvee(X,Y)$ if and only if whenever $P_i \in \mathcal{P}_i$ and $P_{i+1} \in \mathcal{P}_{i+1}$, then there is an arc $(x_i,x_{i+1})$ from a vertex $x_i \in P_i$ to a vertex $x_{i+1} \in P_{i+1}$ if and only if every arc of the form $(x_i, x_{i+1})$ with $x_i \in P_i$ and $x_{i+1} \in P_{i+1}$ is contained in $A(\Gamma)$.
\end{theorem}

We are now ready to prove our main theorem regarding recognition of bicoset digraphs that are $X$-joins using their connection set.  This is the analog of \cite[Theorem 3.3]{BarberDpreprint} for bicoset digraphs.  

\begin{theorem} \label{bicosjoin}
Let $G$ be a group, $i \in \mathbb{Z}_{2}$, $H_i\le K_i\le G$, and $S_{i} \subseteq G$ such that $S_{i} = H_{i}S_{i}H_{i+1}$.  Let $\Gamma = B(G,H_0,H_1;S_0,S_1)$ have arcs, and let $X = \Gamma /{\cal P}(H_i,K_i)$.  For $g\in G$, let $Y_{i,g}$ be the empty digraph on the left cosets of $H_i$ contained in $gK_i$, and $Y = \{Y_{i,g}:i\in\Z_2,g\in G\}$.  Then $\Gamma = \bigvee(X,Y)$ if and only if $S_i = K_iS_iK_{i+1}$.  If $S_i = K_iS_iK_{i+1}$, then $ X = B(G,K_0,K_1;S_0,S_1)$.  Also, if $K_0$ and $K_1$ are chosen to be as large as possible, then $X$ is irreducible.  In addition, if $X$ is weakly connected when $\Aut(X)$ is transitive, then $\Aut(\Gamma)$ is the complete group of natural automorphisms of $\bigvee(X,Y)$. 
\end{theorem}

\begin{proof}
Set ${\cal P} = {\cal P}(H_i,K_i)$.  Suppose $\Gamma \cong \bigvee(X,Y)$.  Let $s_i \in S_i$.  Then $H_i$ is out-adjacent to $s_{i}H_{i+1} \subseteq s_{i}K_{i+1}$.  As $\Gamma \cong \bigvee(X,Y)$, by Theorem \ref{bicoswr}, every vertex of $K_i$ is out-adjacent to every vertex of $s_iK_{i+1}$.  Let $k_i^{-1}\in K_i$ and $k_{i+1}\in K_{i+1}$.  Then $(k_i^{-1}H_i,s_ik_{i+1}H_{i+1} )\in A(\Gamma)$ as $k_i^{-1}H_i$ is a vertex of $K_i$ and $k_{i+1}H_{i+1}$ is a vertex in $K_{i+1}$, and so $k_is_ik_{i+1}\in S_i$.  As $k_i$ and $k_{i+1}$  are arbitrary, $K_is_iK_{i+1}\subseteq S_i$ and as $s_i$ is arbitrary, $S_i$ is a union of $(K_i,K_{i+1})$-double cosets. 

Conversely, suppose $H_{i} \leq K_{i} \leq G$ such that $S_{i} = K_iS_iK_{i+1}$.  Let $(xH_{i},yH_{i+1}) \in A(\Gamma)$, where $x,y\in G$.  Then $x^{-1}y\in S_{i}$, and as $S_{i} = K_iS_iK_{i+1}$, we see  $k_{i}^{-1}x^{-1}yk_{i+1} \in S_{i}$ for every $k_{i}\in K_{i}$ and $k_{i+1}\in K_{i+1}$.  Then $(xk_{i}H_{i}, yk_{i+1} H_{i+1}) \in A(\Gamma)$ for all $k_{i}\in K_{i}$ and $k_{i+1}\in K_{i+1}$.  So every vertex contained in $xK_{i}$ is out-adjacent in $\Gamma$ to every vertex in $yK_{i+1}$ (here $K_{i}$ and $K_{i+1}$ are viewed as unions of cosets of $H_{i}$ and $H_{i+1}$, respectively) for every $(xH_i,yH_{i+1})\in A(\Gamma)$. As $\Gamma[\{i\}\times gK_i]$ has no arcs, by Theorem \ref{bicoswr} $\Gamma = \bigvee(\Gamma/\mP, Y)$.

Suppose that $K_iS_iK_{i + 1} = S_i$.  We now show $X = B(G, K_0, K_1; S_0, S_1)$. First, $B(G,K_0,K_1;S_0,S_1)$ exists and is well-defined by supposition.  By the definition of $B(G,K_0,K_1;S_0,S_1)$, we have $V(X) = {\cal P} = V(B(G,K_0,K_1;S_0,S_1))$. Let $a,b\in G$. Then 
\begin{eqnarray*}
(aK_{i},bK_{i+1})\in A(B(G,K_0,K_1; S_0, S_1)) & {\rm\ if\ and\ only\ if\ } & a^{-1}b \in S_{i}\\
                                               & {\rm\ if\ and\ only\ if\ } & (aH_i,bH_{i+1})\in A(B(G,H_0,H_1;S_0,S_1))\\
                                               & {\rm\ if\ and\ only\ if\ } & (aK_i,bK_{i+1})\in A(X).
\end{eqnarray*}
This shows that $X = B(G, K_0, K_1; S_0, S_1)$.  To see that $\Gamma = \bigvee(X,Y)$, observe that $\bigvee(X, Y)$ has vertex set $V(B(G,H_0,H_1;S_0,S_1))$, and arc set $$\{(ak_iH_i,bk_{i+1}H_{i+1}):(aK_i,bK_{i+1})\in A(X){\rm\ and\ }k_iH_i\in K_i/H_i, k_{i+1}H_{i+1}\in K_{i+1}/H_{i+1}\}.$$  As $(aK_i,bK_{i+1})\in A(B(G,K_0,K_1;S_0,S_1))$, we see that $b = as_i$ for some $s_i\in S_i$, or $s_i = a^{-1}b$.  The arc  $(ak_iH_i,bk_{i+1}H_{i+1})$ is an arc of $\Gamma$ if and only if $k_{i+1}^{-1}a^{-1}bk_{i+1}\in S_i$, which, as $S_i = K_iS_iK_{i+1}$, is true.  Thus $\Gamma = \bigvee(X,Y)$.

Now additionally assume that $K_0$ and $K_1$ are chosen to be as large as possible.  Let ${\cal E}$ be the unworthy partition of $\Gamma$.  By Lemma \ref{P refines B} we have ${\cal E}\prec{\cal B}$.  Let ${\cal E}_i$ be the set of elements of ${\cal E}$ contained in $B_i$.  By Lemma \ref{R and Kernel} we see that ${\cal E}_i$ is the set of orbits of the kernel of the action of $F$ on $B_{i+1}$, where $F$ is the set-wise stabilizer of each block of ${\cal B}$, and is the set of left cosets of some subgroup $H_i\le K_i\le G$. By Lemma \ref{Gamma mod is irreducible} $\Gamma/{\cal E}$ is irreducible.  Also, as $K_iS_iK_{i+1} = S_i$, by the first part of this theorem, we see that $\Gamma = \bigvee(\Gamma/\mP, Y)$.  We then have ${\cal P}\preceq{\cal E}$ as $u R_{\Gamma} v$ for any $u,v \in Y_x$, $x\in X$.  However, if ${\cal P}\not = {\cal E}$, then $\Gamma/{\cal P}$ is not irreducible, and so by the first part of this theorem, there is $K_i < K_i'$ such that $K_i'S_iK_{i+1}' = S_i$, and $K_i$ is not the largest subgroup for which $K_iS_iK_{i+1} = S_{i+1}$, a contradiction.  Thus ${\cal P} = {\cal E}$, and $X = \Gamma/{\cal P}$ is irreducible by Theorem \ref{bipartite join}. 

If, in addition to $K_0$ and $K_1$ being chosen to be as large as possible, so that the unworthy partition of $V(\Gamma)$ is the join partition ${\cal P}(H_i,K_i)$, we also have that if $\Gamma$ is weakly connected when $\Aut(\Gamma)$ is transitive, then by Theorem \ref{xjaut}, we have that $$\Aut(\Gamma/{\cal P}(H_i,K_i)) = \Aut\left(\bigvee(X,Y)\right) = \Aut(\Gamma)$$ is the complete group of natural automorphisms of $\bigvee(X,Y)$.
\end{proof}

Some observations are in order.  First, it is possible that a bicoset digraph $\Gamma = B(G,H_0, H_1;S_0,S_1)$ is isomorphic to an $X$-join of $Y$, where $X = B(G',H_0',H_1';S_0,S_1)$ is a bicoset digraph and $Y$ is a set of empty digraphs, but there is no relationship at all between $G$ and $G'$, as the next example shows.  So the result above does not capture {\it all} the ways a bicoset digraph can be isomorphic to an $X$-join of a set of empty digraphs, but rather only those where $X$ is a quotient of $\Gamma$ using a partition of $V(\Gamma)$ which is a ${\cal P}(H_i,K_i)$-join partition of $V(\Gamma)$.

\begin{example}\label{example 1}
The bicoset graph $\Gamma = B(\Z_5,1,\Z_5;\{0\},\{0\})$ is a star on $6$ vertices (that is, it has six vertices with one vertex of degree $5$ and $5$ vertices of degree $1$), and is isomorphic to the $K_2$-join  of $\{K_1,\bar{K}_5\}$, where $\bar{K}_5$ is the complement of $K_5$.  Let $X = B(\Z_2,1,\Z_2;\{0\}, \{0\})$ (so $X$ is the star on $3$ vertices), and label the vertices of $X$ with elements of $\{x,y,z\}$ such that the unique vertex of $X$ of degree $2$ is $x$.  Set $Y_x = K_1$, $Y_y = \bar{K}_2$, $Y_z = \bar{K}_3$, and $Y = \{Y_x,Y_y,Y_z\}$.  Then $\Gamma$ is isomorphic to the $X$-join of $Y$, and is also a star on $6$ vertices.
\end{example}

\begin{problem}
Determine necessary and sufficient conditions for a bicoset digraph to be isomorphic to an $X$-join with a collection of empty digraphs.
\end{problem}

The next result gives when a bicoset digraph is a wreath product.

\begin{corollary}\label{join is wreath}
Let $G$ be a group, $i \in \mathbb{Z}_{2}$, $H_i\le K_i\le G$, and $S_{i} \subseteq G$ such that $S_{i} = H_{i}S_{i}H_{i+1}$.  Let $\Gamma = B(G,H_0,H_1;S_0,S_1)$ have arcs, and let $X = \Gamma /{\cal P}(H_i,K_i)$.  For $g\in G$, let $Y_{i,g}$ be the empty digraph on the left cosets of $H_i$ contained in $gK_i$, and $Y = \{Y_{i,g}:i\in\Z_2,g\in G\}$.  Then $\Gamma\cong \Gamma/{\cal P}(H_i,K_i)\wr \bar{K}_m$ if and only if $S_i = K_iS_iK_{i+1}$ and $m = [K_0 : H_0] = [K_1 : H_1]$.  Also, if $K_0$ and $K_1$ are chosen to be as large as possible, then $X$ is irreducible.  In addition, if $X$ is weakly connected when $\Aut(X)$ is transitive, then $\Aut(\Gamma) \cong \Aut(X)\wr S_m$. 
\end{corollary}

\begin{proof}
By Theorem \ref{bicosjoin}, we have that $\Gamma\cong\bigvee(X,Y)$.  The digraph $\bigvee(X,Y)\cong \Gamma/{\cal P}\wr\bar{K}_m$ if and only if each $Y_{i,j}$ is isomorphic, which occurs if and only if $m = [K_0 : H_0] = [K_1 : H_1]$.
\end{proof}

Notice that it is possible for a bicoset digraph $\Gamma$ to be simultaneously isomorphic to a wreath product and an $X$-join of empty digraphs that is not written as a wreath product.

\begin{example}\label{example 2}
Let $n\ge 3$ and $G$ a group of order $n$.  Then $K_{n,n}\cong B(G,1,1;G,G)\cong K_2\wr\bar{K}_n$.  Let $\{S_1,S_2\}$ be a partition of $G$ into two subsets of different sizes, $P_1 = \{(0,p_1):p_1\in S_1\}$, $P_2 = \{(0,p_2):p_2\in S_2\}$, and $P_3 = \{(1,g):g\in G\}$. Set ${\cal P}_0 = \{P_1,P_2\}$ and ${\cal P}_1 = \{P_3\}$.  Then ${\cal P} = {\cal P}_0\cup {\cal P}_{1}$ is a partition of $V(B(G,1,1;G,G))$.  Let $X$ be the digraph with $V(X) = {\cal P}$ and arc set $\{(P_1,P_3),(P_3,P_1),(P_2,P_3),(P_3,P_2)\}$, and let $Y_{P_i}$ be the digraph with vertex set $P_i$ and no arcs, $i = 1,2,3$.  Then the vertex set of $B(G,1,1;G,G) = \bigvee(X,Y)$ is not the vertex set of a wreath product of two graphs, and so is not a wreath product as $\vert S_1\vert\not = \vert S_2\vert$ (it is though, of course, {\it isomorphic} to a wreath product).
\end{example}

When $H_0 = H_1 = \{1_G\}$, $B(G,H_0,H_1,S) \cong \text{Haar}(G,S)$, and we have a special case of Theorem \ref{bicosjoin} for Haar graphs.

\begin{corollary}\label{Haar cor}
Let $G$ be a group, $i \in \mathbb{Z}_{2}$, $K_{i} \leq G$, and $S_{i} \subseteq G$. Let $\Gamma = \Haar(G,S_0,S_1)$ have arcs, and let $X = \Gamma /\mP(1,K_i)$. For $g\in G$, let $Y_{i,g}$ be the empty digraph on the left cosets of $gK_i$ and $Y = \{ Y_{i,g} :i\in\Z_2,g \in G \}$. Then $\Gamma = \bigvee(X,Y)$ if and only if $S_i = K_iS_iK_{i + 1}$. If $S_i = K_iS_iK_{i+1}$, then $X = B(G,K_0,K_1;S_0,S_1)$.  Also, if the $K_i$ are chosen to be as large as possible, then $X$ is irreducible.  In addition, if $X$ is weakly connected when $\Aut(\Gamma)$ is transitive, then $\Aut(\Gamma)$ is the group of natural automorphisms of $\bigvee(X,Y)$.
\end{corollary}

The next corollary gives when a Haar digraph is isomorphic to a wreath product.

\begin{corollary}\label{diHaar auto cor}
Let $G$ be a group, $i \in \mathbb{Z}_{2}$, $K_{i} \leq G$, and $S_{i} \subseteq G$. Let $\Gamma = \Haar(G,S_0,S_1)$ have arcs, and let $X = \Gamma /\mP(1,K_i)$. For $g\in G$, let $Y_{i,g}$ be the empty digraph on the left cosets of $gK_i$ and $Y = \{ Y_{i,g} :i\in\Z_2,g \in G\}$. Then $\Gamma = \Gamma / \mP \wr \overline{K}_{m}$ if and only if $S_i = K_iS_iK_{i+1}$ and $m = |K_0| = |K_1|$.  If $S_i = K_iS_iK_{i+1}$ and $\vert K_0\vert = \vert K_1\vert$, then $X = B(G,K_0,K_1;S_0,S_1)$.  Also, if the $K_i$ are chosen to be as large as possible, then $X$ is irreducible.  In addition, if $X$ is weakly connected when $\Aut(\Gamma)$ is transitive, then $\Aut(\Gamma) \cong \Aut(X)\wr S_m$.
\end{corollary}

\bibliography{References}{}
\bibliographystyle{amsplain}

\end{document}